\documentclass[11pt]{article}

\usepackage{float}
\usepackage{graphicx,color}
\usepackage{epsfig}
\usepackage{epstopdf,mathtools}
\usepackage{amsmath,amssymb,amsthm}
\usepackage{subfigure}
\usepackage{dsfont}
\usepackage{psfrag,multirow}
\usepackage[ruled,noline]{algorithm2e}
\usepackage{algpseudocode}
\usepackage{authblk}
\usepackage[pagewise]{lineno}
\def\1{\mathbf{1}}

\def\L{\mathcal{L}}

\graphicspath{{figures/}}

\newtheorem{theorem}{Theorem}[section]

\newtheorem{lemma}[theorem]{Lemma}

\newtheorem{remark}[theorem]{Remark}
\newtheorem*{assumption*}{\assumptionnumber}
\providecommand{\assumptionnumber}{}
\makeatletter
\newenvironment{assumption}[2]
 {%
  \renewcommand{\assumptionnumber}{Assumption #1-$\mathcal{#2}$}%
  \begin{assumption*}%
  \protected@edef\@currentlabel{#1-$\mathcal{#2}$}%
 }
 {%
  \end{assumption*}
 }
\makeatother

\floatstyle{ruled}
\newfloat{Algorithm}{tbp}{lop}[section]

\newcommand{\argmin}{\operatornamewithlimits{argmin}}
\begin{document}
\title{Stochastic Optimization using Polynomial Chaos Expansions}
\author{Tuhin Sahai\thanks{Raytheon Technologies Research Center, Berkeley, CA, USA. {\tt tuhin.sahai@gmail.com}}}
\maketitle
\begin{abstract}
Polynomial chaos based methods enable the efficient computation of output variability in the presence of input uncertainty in complex models. Consequently, they have been used extensively for propagating uncertainty through a wide variety of physical systems. These methods have also been employed to build surrogate models for accelerating inverse uncertainty quantification (infer model parameters from data) and construct transport maps. In this work, we explore the use of polynomial chaos based approaches for optimizing functions in the presence of uncertainty.  These methods enable the fast propagation of uncertainty through smooth systems. If the dimensionality of the random parameters is low, these methods provide orders of magnitude acceleration over Monte Carlo sampling. We construct a generalized polynomial chaos based methodology for optimizing smooth functions in the presence of random parameters that are drawn from \emph{known} distributions. By expanding the optimization variables using orthogonal polynomials, the stochastic optimization problem reduces to a deterministic one that provides estimates for all moments of the output distribution. Thus, this approach enables one to avoid computationally expensive random sampling based approaches such as Monte Carlo and Quasi-Monte Carlo. In this work, we develop the overall framework, derive error bounds, construct the framework for the inclusion of constraints, analyze various properties of the approach, and demonstrate the proposed technique on illustrative examples.
\end{abstract}
\section{Introduction}
Uncertainty quantification (UQ) is a popular area of research that focuses on propagating uncertainty through complex dynamic systems (typically represented by ordinary or partial
differential equations). Typical approaches for propagating uncertainty include Monte Carlo~\cite{McQMC}, Quasi-Monte Carlo methods~\cite{QMC}, and
importance sampling~\cite{Importance_Sampling,budhiraja2019analysis}. These methods are based on the sampling of the underlying input probability distributions, and consequently, are standard techniques for estimating output
uncertainty. In addition to sampling based techniques, over the last couple of decades, the UQ community has actively pursued the development of non-sampling approaches such as
response surface~\cite{Allen2009,TuhinPoly} and polynomial chaos~\cite{Wiener} based methodologies. Polynomial chaos methods involve the expansion of the stochastic variable of
interest using an orthogonal basis associated with the underlying distribution. This step is typically followed by a projection computation that exploits the aforementioned orthogonality. Let us now briefly discuss the conditions under which various sampling and non-sampling methods are found to be useful.

Sampling based methods for UQ rely on generating samples in parameter space, propagating the points through the system (evolving the points forward using numerical integration in the 
case of dynamical systems), and computing statistics of the first few moments of the output distribution. Monte Carlo methods involve generating random points that
correspond to independent trials. Note that the convergence of Monte Carlo is guaranteed by the strong law of large numbers. If one generates $N$ independent samples, the error in the mean
estimate converges as $O(N^{-1/2})$~\cite{QMC}. The advantage of Monte Carlo based sampling is that the convergence is independent of the
number of random parameters. Quasi-Monte Carlo based approaches, on the other hand, involve the generation of points using deterministic schemes. In particular, these points are generated
using low-discrepancy sequences (low-discrepancy sequences have the property that, in the limit, the fraction of points that fall into an arbitrary set $\mathbb{A}$ is equal to the
measure of $\mathbb{A}$). Quasi-Monte Carlo methods have a convergence rate of $\log^{d}(N)/N$, where $d$ is the dimensionality of the input random parameter vector. Note that
Monte Carlo and Quasi-Monte Carlo methods are routinely used in stochastic optimization for computing sample average approximations (SAA)~\cite{Cit:SAA_opt,Cit:SAA_opt2}.

As mentioned previously, polynomial chaos methods are non-sampling methods that rely on expanding the output random variables using an orthogonal polynomial basis with respect
to the prior distribution~\cite{PolyReview}. For example, if the joint probability distribution (prior) that captures the uncertainty in the random input variables is Gaussian, one uses the
Hermite basis~\cite{Orthopoly:book}. Similarly, if the uncertainty is captured by the uniform distribution, then one uses Legendre polynomials~\cite{Orthopoly:book}. One can generate
orthogonal polynomials for arbitrary distributions, for more details see~\cite{PolyReview}. The primary advantage of the polynomial chaos based approach is that it provides
exponential convergence for smooth processes with finite variance~\cite{cameron1947orthogonal}. This remarkable convergence result is obtained by invoking the Cameron-Martin theorem. Note that this
approach only works for random variables with finite variance where the underlying probability measure can be uniquely determined by its moments. A major drawback of the approach is that it suffers
from a curse of dimensionality. This curse of dimensionality will be discussed in greater detail in subsequent sections. 

Our paper is organized as follows: we start by introducing the polynomial chaos approach. We then show how one can use this method to
 efficiently compute the solutions of stochastic optimization problems and derive associated error estimates. We also explore properties associated with the
 polynomial chaos transformation and its consequences for stochastic optimization. We demonstrate the approach on illustrative examples including a real-world human-machine task scheduling problem. This new methodology for stochastic
 optimization is compared to sampling based methods. Although in~\cite{poles2009polynomial}, the authors invoke polynomial chaos for multiobjective optimization, they neither integrate these expansions into the stochastic optimization framework nor derive approximation bounds. We note that in~\cite{Cit:grid_PC}, the authors use polynomial chaos based surrogate models for stochastic optimization of the power grid.

\section{Introduction to Polynomial Chaos}
Starting with a complete probability space $\Gamma$ given by
$(\Omega, \mathcal{F},\mathbb{P})$, where $\Omega$ is the sample
space, $\mathcal{F}$ is the $\sigma$-algebra on $\Omega$ and
$\mathbb{P}$ is a probability measure, let $L_{2}(\Gamma,X)$
denote the Hilbert space of square-integrable,
$\mathcal{F}$-measurable, $X$-valued random elements. Then one can,
in general, define a polynomial chaos basis
$\{\psi_k(\lambda(\omega))\}$, where $\lambda(\omega)$ is a
random vector, $\omega \in \Omega$, and $k = (k_1,k_2,\dots)$ is a
vector of non-negative indices. We denote the probability
density function of the random vector $\lambda$ by
$\rho(\lambda)$.

Generalized polynomial chaos (gPC)~\cite{BeyondWienerAskey} provides a framework for
representing second-order stochastic processes $\kappa\in
L_{2}(\Gamma,X)$ for arbitrary distributions of $\lambda$ by using the following expansion:
\begin{equation}
\kappa(\lambda) =
\displaystyle\sum_{|k|=0}^{\infty}a_k \psi_k (\lambda),
\label{eq:expan1}
\end{equation}
where $|k| = \sum_{i} k_{i}$ is the sum of the indices of $k$ and
$\psi_k (\lambda)$ are orthonormal polynomials on $\Gamma$
with respect to $\rho(\lambda)$. Restricting our formalism to Euclidean spaces (relevant for this work) the orthonormality is given by,
\begin{equation}
\displaystyle\int_{\mathbb{R}^p}
\rho(\lambda)\psi_i(\lambda)\psi_k(\lambda)d\lambda =
\delta_{ik}, \label{eq:ortho}
\end{equation}
where $\delta_{ik}$ is the Kronecker delta product.
Depending on $\rho(\lambda)$, one can generate an appropriate
orthogonal basis for representing~$\kappa(\lambda)$. As mentioned earlier, if $\rho$ is
Gaussian, then the appropriate polynomial chaos basis is the set
of Hermite polynomials; if $\rho$ is the uniform distribution,
then the basis is the set of Legendre polynomials. For details
on the correspondence between distributions and polynomials
see~\cite{PolyReview,Ogura}. A framework for generating polynomials
for arbitrary distributions has been developed
in~\cite{BeyondWienerAskey}. The advantage of using polynomial chaos is that it provides exponential convergence for smooth processes with finite variance~\cite{cameron1947orthogonal}. However, the approach suffers from a curse of dimensionality, rendering it infeasible for problems with more than a handful of random parameters. In particular, one typically truncates the order of expansion in Eqn.~\ref{eq:expan1} (to $r$ terms). One then uses the orthogonality property in Eqn.~\ref{eq:ortho} to project the original equation onto the different coefficients $a_k$ in Eqn.~\ref{eq:expan1}~\cite{BeyondWienerAskey}. Typically, low order truncations are found to capture the uncertainty in smooth systems~\cite{cameron1947orthogonal} (as long as the underlying probability measure can be uniquely determined by its moments~\cite{Cit:ernst2012convergence}). If the order of expansion is $r$ and the dimensionality of the uncertain parameters is $p$, then the number of terms one gets is,
\begin{align}
d\frac{(r+p)!}{r!p!},
\label{eq:curse_dim}
\end{align}
where $d$ is the dimensionality of $x$ (since we assume that $x\in\mathbb{R}^d$).
To mitigate the above curse of dimensionality, sparse grid techniques~\cite{Webster2007, Nobile2008, Zabaras2008}, iterative methods~\cite{surana_uq, Cit:sahai2010cdc, Sahai2012, Klus2011}, regression based algorithms~\cite{blatman2010adaptive,blatman2011adaptive}, hierarchical methods~\cite{ma2009adaptive}, and dimensionality reduction based techniques~\cite{ma2011kernel,marzouk2009dimensionality} have been developed. We now explore the use of these methods for optimizing functions in the presence of uncertainty.

\section{Stochastic Optimization using Polynomial Chaos}
Without loss of generality, we assume that the optimization problem is posed in the form,
\begin{align}
\displaystyle\min_{x \in \mathbb{R}^d}& \,\, f(x,\lambda), \nonumber\\
\text{subject to}\quad g_{i}(x,\lambda) &\leq 0,\,\, i=1,\hdots,m\nonumber \\
h_{j}(x,\lambda) &=0,\,\, j=1,\hdots,n,
\label{eq:overall}
\end{align}
where $\lambda$ is a vector of random variables drawn from the probability distribution $\rho(\lambda)$ and $f(x,\lambda), h(x,\lambda),$ and $g(x,\lambda)$ are smooth functions of $x$ and $\lambda$. Here $m$ and $n$ are the number of inequality and equality constraints respectively.

A host of algorithms have been developed in the areas of stochastic programming~\cite{Cit:stochastic} and distributionally robust optimization~\cite{Cit:distributionally,Cit:distributionally2} to tackle this problem. However, to the best of our knowledge, none of these methods or their variants exploit the exponential convergence offered by polynomial chaos based approaches~\cite{cameron1947orthogonal}.

Typically the above set of equations (Eqn.~\ref{eq:overall}) are solved in either expectation or worst case. In the case of expectation minimization, the above optimization is converted to,
\begin{align}
\displaystyle\min_{x\in\mathbb{R}^d}\,\,\mathbb{E}\left[f(x,\lambda)\right]
\label{eq:expectation}
\end{align}
which is usually computed using \emph{expensive} Monte Carlo computations~\cite{Cit:Montecarlo_opt}. In this work, we exploit polynomial chaos expansions to approximate the optimal solution (in expectation) without resorting to expensive Monte Carlo simulations. The advantage of our approach is that one can compute several moments of the optimal solution (mean, variance, and other higher order moments) through a single optimization computation, without the need for expensive sampling. Note that the accuracy of computed moments depend on the order of expansion for the variables.

We now explore the use of polynomial chaos in the context of stochastic optimization. For simplicity, assume that the optimal solution for the problem in Eqn.~\ref{eq:overall}, in the absence of constraints, is, 
\begin{align}
x^{*}(\lambda) = \displaystyle\argmin_{x}f(x,\lambda).
\label{eq:opt_simple}
\end{align}
Although constraints in Eqn.~\ref{eq:overall} have been ignored for the moment, the constrained optimization case will be revisited in Section~\ref{sec:constraints}. We now approximate the optimization variables in terms of the orthogonal polynomials $\psi_{k}(\lambda)$ as follows,
\begin{align}
x(\lambda) \approx \displaystyle\sum_{k=0}^{r} a_{k}\psi_{k}(\lambda),
\label{eq:poly_exp}
\end{align}
which results in the following approximate optimization problem,
\begin{align}
 \displaystyle\left[a_{0}^{*},\hdots,a_{r}^{*}\right] = \displaystyle\argmin_{\left[a_{0},\hdots,a_{r}\right]}f(\displaystyle\sum_{k=0}^{r} a_{k}\psi_{k}(\lambda),\lambda).
\label{eq:expanded_opt}
\end{align}

Using expectations and interchanging the integral and minimization gives (the bounds on the error due to this interchange are derived in section~\ref{sec:interchange}),
\begin{align}
\displaystyle\left[a_{0}^{*},\hdots,a_{r}^{*}\right] \approx \displaystyle\argmin_{\left[a_{0},\hdots,a_{r}\right]}\int_{\mathbb{R}^p}f(\displaystyle\sum_{k=0}^{r} a_{k}\psi_{k}(\lambda),\lambda)\rho(\lambda)d\lambda.
\label{eq:final_min}
\end{align}
Denoting the integral term as $F(a_{0},\hdots,a_{r})$ gives,
\begin{align}
\displaystyle\left[a_{0}^{*},\hdots,a_{r}^{*}\right] \approx \displaystyle\argmin_{\left[a_{0},\hdots,a_{r}\right]}F(a_{0},\hdots,a_{r}).
\label{eq:final_min2}
\end{align}
Typically $F(a_{0},\hdots,a_{r})$ reduces to a simple form due to the orthogonal properties of the basis. Moreover, the coefficients $a_{k}$ in the expansion can be used to compute the moments of $x^{*}$~\cite{Xiu2003}. In particular, the mean $\mu_{0}(x^{*}) = a_{0}$, standard deviation $\mu_{1}(x^{*}) = \sqrt{a_{1}^2 + a_{2}^2 + \hdots + a_{r}^2}$, and so on. This expansion is guaranteed to converge to the correct answer as long as the associated moment generating function converges to the input distribution~\cite{Cit:ernst2012convergence}. As mentioned previously, the dimensionality of the optimization in Eqn.~\ref{eq:final_min2} is much higher than the one in Eqns.~\ref{eq:overall} and~\ref{eq:expectation}. An important distinction of the optimization in Eqn.~\ref{eq:expectation} from the one in Eqn.~\ref{eq:final_min2} is that, although both problems are deterministic, the latter retains information regarding higher-order statistics of $x^{*}$ in the form of the coefficients. In particular, by substituting Eqn.~\ref{eq:poly_exp} in the expressions below, one can compute the moments of $x^{*}$ in terms of $\left[a_{0}^{*},\hdots,a_{r}^{*}\right]$ using,
\begin{align}
\mu_{0} &= \displaystyle\int_{\mathbb{R}^p}x^{*}\rho(\lambda)d\lambda, \nonumber\\
\vdots&\nonumber\\
\mu_{k} &=\displaystyle\int_{\mathbb{R}^p}(x^{*}-\mu_{0})^k\rho(\lambda)d\lambda,
\label{eq:moment_eqns}
\end{align}
where $x^{*}$ is approximated using Eqns.~\ref{eq:poly_exp} and~\ref{eq:final_min2}.
\subsection{Convergence of Polynomial Chaos}
The results by Cameron and Martin~\cite{cameron1947orthogonal} show that a square integrable functional on the set of continuous functions with compact support can be expanded in a convergent series of Hermite polynomials in a countable sequence of Gaussian random variables. This result was extended to the generalized polynomial chaos setting for arbitrary distributions with finite variance under the condition that the underlying probability measure is uniquely determined by its moments~\cite{Cit:ernst2012convergence}.

In our framework, since $x^{*}$ is effectively treated as a random variable, the theorems and proofs from~\cite{Cit:ernst2012convergence} are applicable. We extract the primary results from~\cite{Cit:ernst2012convergence} and present them to the reader for completeness.

As detailed in~\cite{Cit:ernst2012convergence}, the primary assumptions for $x^{*}$ are as follows,
\begin{assumption}{1}{F}\label{assump1}
$x^{*}$ possesses finite moments of all orders, i.e., $\int x^{*}(\lambda)d\lambda<\infty$ for all $k$.
\end{assumption}
\begin{assumption}{2}{F}\label{assump2}
The distribution function $P(x^{*}\leq\xi)$ is continuous.
\end{assumption}
The primary theorems from~\cite{Cit:ernst2012convergence} that guarantee convergence are listed below.
\begin{theorem}
The sequence of orthogonal polynomials associated with the random variable $x^{*}$ satisfying assumptions~\ref{assump1} and~\ref{assump2} is dense in the Hilbert space $L_{2}(\Omega, \sigma, P)$ if and only if the moment problem is uniquely solvable for its distribution.
\end{theorem}
\begin{proof}
See~\cite{Cit:ernst2012convergence}.
\end{proof}
The conditions for convergence are given by the theorem below, as shown in~\cite{Cit:ernst2012convergence}.
\begin{theorem}
If one of the following conditions for $P(x^{*}\leq\xi)$ satisfying assumptions~\ref{assump1} and~\ref{assump2} is valid, then the moment problem is uniquely solvable and the set of polynomials associated with $x^{*}$ is dense in $L_{2}(\Omega, \sigma, P)$.
\begin{enumerate}
\item $P(x^{*}\leq\xi)$ has compact support.
\item The moment sequence of $\{\mu_{k}\}$ of the distribution satisfies
\begin{align*}
\displaystyle\lim\inf_{k\rightarrow\infty}\frac{\sqrt[2k]{\mu_{2k}}}{2k}<\infty \quad\text{or}\quad \displaystyle\sum_{k=0}^{\infty}\frac{1}{\sqrt[2k]{\mu_{2k}}}=0.
\end{align*}
\item The random variable is exponentially integrable
\begin{align*}
\int_{\mathbb{R}}\exp(a|\lambda|)P(d\lambda) < \infty.
\end{align*}
\item If the distribution has a symmetric, differentiable and strictly positive density $l_{x^{*}}$ and for a real number $\lambda_{0}$ there holds
    \begin{align*}
    \int_{-\infty}^{\infty}-\frac{\log(l_{x^{*}}(\lambda))}{1+\lambda^{2}}d\lambda=\infty \quad\text{and}\quad -\frac{\lambda l^{'}_{x^{*}}(\lambda)}{l_{x^{*}}} \nearrow \infty\,\, (\lambda\rightarrow\infty,\lambda\geq\lambda_{0}).
    \end{align*}
\end{enumerate}
\end{theorem}
\begin{proof}
See~\cite{Cit:ernst2012convergence}.
\end{proof}
As long as the above assumptions and conditions are satisfied,   the expansion in Eqn.~\ref{eq:poly_exp} is guaranteed to converge to the solution for the problem in Eqn.~\ref{eq:final_min2}. We now analyze the error incurred as a result of the interchange of the integral and minimum operators in Eqn.~\ref{eq:final_min}.
\subsection{Error due to Interchange of Integral and Minimum operators}
\label{sec:interchange}
A key approximation in the derivation for Eqn.~\ref{eq:final_min2} was the interchange of integral and minimum operators in Eqn.~\ref{eq:final_min}. We now bound the error introduced due to the interchange of the two operators.
\begin{lemma}
If $f(x,\lambda)$ in Eqn.~\ref{eq:opt_simple} is Lipschitz continuous with respect to $x$ with a Lipschitz constant $L$, then the error due to the interchange of the integral and minimization in Eqn.~\ref{eq:final_min} has the following bound,
\begin{equation}
\left|\mathbb{E}\left[\displaystyle\min_{x}\,\,f(x,\lambda)\right]-\displaystyle\min_{x}\,\,\mathbb{E}\left[f(x,\lambda)\right]\right| \leq L \displaystyle\int_{\mathbb{R}^p}\left|\hat{p}(\lambda) - q\right|\rho d\lambda, \nonumber
\end{equation}
where $\hat{p}(\lambda)= \displaystyle\argmin_{x} f(x,\lambda)$ and $q=\displaystyle\argmin_{x}\,\,\mathbb{E}\left[f(x,\lambda)\right]$.
\end{lemma}
\begin{proof}
The above definition of $\hat{p}(\lambda)$ gives,
\begin{align}
\mathbb{E}\left[\displaystyle\min_{x}\,\,f(x,\lambda)\right] = \displaystyle\int_{\mathbb{R}^p}f(\hat{p}(\lambda),\lambda)\rho d\lambda,
\label{eq:expect_expand}
\end{align}
 and using the above definition of $q$ gives,
we get,
\begin{align}
\displaystyle\argmin_{x}\,\,\mathbb{E}\left[f(x,\lambda)\right] = \displaystyle\int_{\mathbb{R}^p}f(q,\lambda)\rho d\lambda.
\label{eq:min_exp}
\end{align}
Note that the error due to the interchange in Eqn.~\ref{eq:final_min} is given by,
\begin{align}
\left|\mathbb{E}\left[\displaystyle\min_{x}\,\,f(x,\lambda)\right]-\displaystyle\min_{x}\,\,\mathbb{E}\left[f(x,\lambda)\right]\right|.
\end{align}
Note that since $f(x,\lambda)$ is Lipschitz continuous with respect to $x$ with a constant $L$, we get,
\begin{align}
\left|\mathbb{E}\left[\displaystyle\min_{x}\,\,f(x,\lambda)\right]-\displaystyle\min_{x}\,\,\mathbb{E}\left[f(x,\lambda)\right]\right| &= \left|\displaystyle\int_{\mathbb{R}^p}(f(\hat{p}(\lambda),\lambda)  - f(q,\lambda))\rho d\lambda\right|, \nonumber\\
&\leq \displaystyle\int_{\mathbb{R}^p}\left|f(\hat{p}(\lambda),\lambda)  - f(q,\lambda)\right|\rho d\lambda, \nonumber \\
&\leq L\displaystyle\int_{\mathbb{R}^p}\left|\hat{p}(\lambda) - q\right|\rho d\lambda
\end{align}
\end{proof}
As can be seen above, the bound depends on the $\rho$ weighted deviation of $\hat{p}(\lambda)$ from $q$. Here $\hat{p}(\lambda)$ is the optimal solution in Eqn.~\ref{eq:opt_simple} parameterized by $\lambda$ and $q$ is the optimal solution in Eqn.~\ref{eq:expectation} for the expected value of $f(x,\lambda)$. In particular, if $\hat{p}(\lambda)$ deviates from $q$ in the tail of $\rho$, the error due to the interchange is expected to be minimal. Moreover, the Lipschitz constant $L$ bounds the variation of $f(x,\lambda)$ as a function of the argument. Thus, if $L$ is small, the above bound is again expected to be small. Note that error is only expected to be $0$ if $\hat{p}(\lambda)=q \,\,\forall\lambda$.

We now analyze different properties related to the polynomial chaos based stochastic optimization approach.

\subsection{Preservation of convexity}
Convexity is an important property that is frequently exploited in optimization~\cite{Cit:boyd2004convex}. The optimization of convex functions over convex domains is tractable using well-known polynomial time algorithms. Non-convex optimization, on the other hand, is in general NP-hard~\cite{sahai2020dynamical}. Most generic optimization software are able to efficiently compute globally optimal solutions in convex settings. Thus, if the underlying function $f(x,\lambda)$ is convex, it is important that the polynomial chaos based approach preserve convexity. Therefore, we now explore the impact of the expansion on the convexity properties of $f(x,\lambda)$.
\begin{lemma}
If $f(x,\lambda)$ in Eqn.~\ref{eq:opt_simple} is convex for all values of $\lambda$, then $F(a_{1},\hdots,a_{r})$ in Eqn.~\ref{eq:final_min2} is also convex.
\end{lemma}
\begin{proof}
For brevity, let $F(\vec{a}) = F(a_{1},\hdots,a_{r})$. Taking $\vec{a}=\theta \vec{b} + (1-\theta)\vec{c}$ implies that,
\begin{align}
a_{k} = \theta b_{k} + (1-\theta)c_{k},\nonumber
\label{eq:ai}
\end{align}
and dropping the argument of $\psi_{k}(\lambda)$ (for simplicity) gives,
\begin{align}
f(\theta\displaystyle\sum_{k}b_{k}\psi_{k} + (1-\theta)\displaystyle\sum_{k}c_{k}\psi_{k},\lambda)\leq \theta f(\displaystyle\sum_{k}b_{k}\psi_{k},\lambda) + (1-\theta)f(\displaystyle\sum_{k}c_{k}\psi_{k},\lambda),
\end{align}
since $f(x,\lambda)$ is a convex function. Note that one uses the polynomial expansion outlined in Eqn.~\ref{eq:poly_exp} to obtain the above expression. Using the above expression in conjunction with the fact that $\rho(\lambda)\geq 0$ gives,
\begin{align}
\displaystyle\int_{\mathbb{R}^n}f(\theta\displaystyle\sum_{k}b_{k}\psi_{k} + (1-\theta)\displaystyle\sum_{k}c_{k}\psi_{k},\lambda)\rho(\lambda)d\lambda
\leq \theta & \displaystyle\int_{\mathbb{R}^n}f(\displaystyle\sum_{k}b_{k}\psi_{k},\lambda)\rho(\lambda)d\lambda \nonumber\\
&+ (1-\theta)\displaystyle\int_{\mathbb{R}^n}f(\displaystyle\sum_{k}c_{k}\psi_{k},\lambda)\rho(\lambda)d\lambda ,
\end{align}
which can be rewritten as,
\begin{align}
F(\vec{a}) \leq \theta F(\vec{b}) + (1-\theta)F(\vec{c}),
\end{align}
or,
\begin{align}
F(\theta \vec{b} + (1-\theta)\vec{c}) \leq \theta F(\vec{b}) + (1-\theta)F(\vec{c}).
\end{align}
The above inequality proves that $F(\vec{a})$ is convex.
\end{proof}
The preservation of convexity is analogous to the preservation of Hamiltonian structure by the polynomial chaos framework~\cite{pasini2013polynomial,pasini1polynomial}.
\begin{remark}
It is easy to show that if the underlying cost function $f(x,\lambda)$ is a homogeneous polynomial~\cite{cox2013ideals}, then the resulting cost function $F(\vec{a})$ is also homogeneous. This can be shown by using the fact that the right hand side of the expansion $x=\displaystyle\sum_{k=0}^{r}a_{k}\psi_{k}(\lambda)$ is effectively a dot product between the vectors $\left[a_{0},a_{1},\hdots,a_{r}\right]$ and $\left[H_{0},H_{1},\hdots,H_{r}\right]$. Since this form is linear in $\vec{a}$, any homogeneous polynomial will also be homogeneous with the same degree as the original system. In a similar fashion one can show that the transformation preserves the sum-of-squares (S.O.S.) hierarchy~\cite{lasserre2007sum}. In particular, since any sum-of-squares polynomial can, without loss of generality, be written in the following quadratic form, $z^TQz$ where $z=[1, x_{1}, x_{2},\hdots,x_{1}^2,x_{1}x_{2},\hdots x_{m}^d]^T$. When one expands each variable using the orthogonal polynomials, it can be shown that the quadratic structure is preserved, implying that the S.O.S. structure is preserved by the transformation.
\end{remark}

\section{Inclusion of constraints \& dual formulations}\label{sec:constraints}
In the previous section, we omitted the constraints in Eqn.~\ref{eq:overall} for simplicity. The imposition of constraints can be addressed using the Lagrangian framework~\cite{Cit:boyd2004convex} as follows,
\begin{align}
\mathcal{L}(x,u,v,\lambda) = f(x,\lambda) + \displaystyle\sum_{i=1}^{m}u_{i}g_{i}(x,\lambda)  + \displaystyle\sum_{j=1}^{n}v_{j}h_{j}(x,\lambda).
\label{eq:lagrangian}
\end{align}
Using the polynomial expansion in Eqn.~\ref{eq:poly_exp} and integrating with respect to $\rho(\lambda)d\lambda$ gives,
\begin{align}
\tilde{\mathcal{L}}(\vec{a},u,v) = F(\vec{a}) + \displaystyle\sum_{i=1}^{m}u_{i}G_{i}(\vec{a}) + \displaystyle\sum_{j=1}^{n}v_{j}H_{j}(\vec{a}),
\label{eq:lagrangian_poly}
\end{align}
where $\vec{a} = \left[a_{0},\hdots,a_{r}\right]$. Both $G_{i}$ and $H_{j}$ are the terms obtained by integrating $g_{i}$ and $h_{j}$ weighted by $\rho(\lambda)$ (note that the definition is similar to the one for $F$ in Eqn.~\ref{eq:final_min2}). It is easy to show that as a consequence of $\rho\geq0$, the inequality and equality constraints are preserved i.e. $H_{j}(\vec{a}) =0$ and $G_{i}(\vec{a}) \leq 0$. Thus, $\tilde{\mathcal{L}}(u,v)$ in the above equation is the corresponding Lagrangian for the stochastic problem. Now one can define the Lagrange dual function as follows,
\begin{align}
\hat{\L}(u,v) = \displaystyle\min_{\vec{a}}\tilde{\mathcal{L}}(\vec{a},u,v).
\label{eq:dual1}
\end{align}
Note that Eqn.~\ref{eq:overall} is the primal form. It can be shown that $\hat{\L}(u,v) \leq F^{*}$, where $F^{*}$ is the optimal solution of primal formulation~\cite{Cit:boyd2004convex}. Thus, one gets a dual optimization problem of the form,
\begin{align}
\displaystyle\max_{u,v}\,\, &\hat{\L}(u,v), \nonumber \\
\text{subject to}\quad u &\geq 0.
\label{eq:dual2}
\end{align}
The difference between $F^{*}$ and $\displaystyle\max_{u,v}\,\hat{\L}(u,v)$ is known as the duality gap. The condition $F^{*} = \displaystyle\max_{u,v}\,\hat{\L}(u,v)$ is known as the strong duality. Both the conditions hold for the stochastic setting considered here. One can also derive the corresponding Karush-Kuhn-Tucker (KKT) conditions that guarantee optimality~\cite{cit:karush,cit:kuhn} of the form,
\begin{align}
\nabla F(\vec{a}) + \displaystyle\sum_{i=1}^{m}u_{i}\nabla G_{i}(\vec{a}) + \displaystyle\sum_{j=1}^{n}v_{j}\nabla H_{j}(\vec{a}) &= 0,\nonumber
\end{align}
along with the following constraints,
\begin{align}
G_{i}(\vec{a}) &\leq 0, \,\, i=1,\hdots,m\nonumber \\
H_{j}(\vec{a}) &= 0,j=1,\hdots,n\nonumber \\
u_{i} &\geq 0 \,\, i=1,\hdots,m\nonumber \\
u_{i} G_{i}(\vec{a}) &= 0,\,\, i=1,\hdots,m.\nonumber \\
\label{eq:kkt}
\end{align}
The above conditions are necessary for optimality. For convex problems, the KKT conditions are also sufficient for optimality. Thus, by computing the KKT conditions one can obtain the optimal solution (condition on strong duality e.g. Slater's condition). We refer the reader to~\cite{Cit:boyd2004convex} for further details.

We note that, in general, a Lagrange multiplier approach can, under certain conditions, convert minima (or maxima) into saddle points as shown below. Saddles points have both increasing as well as decreasing directions on the energy landscape (for further information about saddle points we refer the reader to~\cite{Gucken:book}).
\begin{lemma}
Let $x^{*}$ be a minimum (similar argument holds for maxima) to the following 1-dimensional problem,
\begin{align}
\displaystyle\min_{x \in \mathbb{R}^d}& \,\, f(x,\lambda), \nonumber\\
h(x,\lambda) &=0,\,\, .
\label{eq:overall_sad}
\end{align}
Then $x^{*}$ becomes a saddle point in the Lagrangian formulation,
\begin{align}
\displaystyle\min_{x,v}&\,\,\mathcal{L}(x,v,\lambda)  \nonumber\\
\mathcal{L}(x,v,\lambda) &= f(x,\lambda) + vh(x,\lambda),
\label{eq:lagrangian_sad}
\end{align}
if $\nabla h \neq 0$.
\end{lemma}
\begin{proof}
For Eqn.~\ref{eq:lagrangian_sad}, the optimality conditions are,
\begin{align}
\nabla f(x,\lambda) + v\nabla h(x,\lambda) &=0,\nonumber\\
\mathcal{L}_{v} = h(x,\lambda) &=0.
\label{eq:opt_sad}
\end{align}
Let $x^{*}$ satisfy the above equations, then the stability of the point is determined by the eigenvalues of the Jacobian matrix,
\begin{align}
J = \begin{bmatrix}
\nabla^{2}f + v\nabla^{2}h & \nabla h \\
\nabla h  & 0
\end{bmatrix}.
\end{align}
The characteristic equation for the eigenvalues $\eta_1,\eta_{2}$ are given by,
\begin{align}
\eta^{2} - \eta(\nabla^{2}f + v\nabla^{2}h) - (\nabla h)^2 = 0.
\end{align}
Thus, this implies that the product of the eigenvalues $\eta_{1}\eta_{2} = - (\nabla h)^2 $ can either be negative or zero. In the case the product is non-zero, it is easy to see that $x^{*}$ is a saddle (since it must have one positive eigenvalue). We note that $\nabla h = 0$ corresponds to $h = c$ which is a trivial constraint and the minimization problem is rendered superfluous.
\end{proof}

To address the above issue, we intend to explore the future integration of the polynomial chaos framework with barrier~\cite{cit:karmarkar} and penalty methods~\cite{cit:bertsekas} for optimization.

\section{Results}
We now demonstrate the polynomial chaos based stochastic optimization framework on a few illustrative optimization examples. We start with a simple 1-dimensional quadratic unconstrained optimization problem. We increase the problem complexity by illustrating the approach on a standard non-convex 2-dimensional optimization problem. We finally include constraints and demonstrate the approach on a  complicated human-machine task scheduling problem that is inspired from real-world aerospace applications.
\subsection{Simple 1-Dimensional Optimization Example}
We now demonstrate the proposed approach on a simple illustrative optimization problem. Consider an objective function of the form,
\begin{align}
\displaystyle\min_{x} (1+\lambda)x^2 + x,
\label{eq:simple_ex}
\end{align}
where $\lambda$ is a normally distributed random variable with mean $\mu_0(\lambda) = 0.0$ and standard deviation $\mu_{1}(\lambda) = 0.1$. Since $\lambda$ is normally distributed, we expand $x$ in terms of Hermite polynomials~\cite{BeyondWienerAskey}. Expanding $x$ to the second order results in the following expression,
\begin{align}
x\approx a_{0}\psi_{0} + a_{1}\psi_{1} + a_{2}\psi_{2},
\label{eq:ex_exp}
\end{align}
We now perform the steps outlined in Eqns.~\ref{eq:expanded_opt} to~\ref{eq:final_min2}, which results in the following optimization problem,
\begin{align}
\displaystyle\min_{\left[a_{0},a_{1},a_{2}\right]} a_{0}^2 + a_{1}^2 + a_{2}^2
 + \mu_{1}(\lambda)\left[2a_{0}a_{1} + 4a_{1}a_{2}\right] + a_{0}.
 \label{eq:simple_pc}
 \end{align}
 Solving eqn.~\ref{eq:simple_pc} using standard quadratic programming solvers yields a mean of $\mu_{0}(x^{*}) \approx -0.505$ and standard deviation $\mu_{1}(x^{*}) = 0.054$. This compares favorably with $1000$ Monte Carlo samples for estimating the mean $\mu_{0}(x^{*}) = -0.508$ and standard deviation $\mu_{1}(x^{*})=0.055$. The above demonstration illustrates how a single optimization computation can yield moments that are close to the statistics of the optimal solution. Thus, the polynomial chaos approach translates into significant computational savings over sampling based methodologies.
 
 \subsection{Two dimensional example}
Consider the Himmelblau test function~\cite{himmelblau2018applied} for optimization given by,
\begin{align}
    f(x_1,x_2) & = (x_{1}^2 + x_{2} - 11)^2 + (x_{1} + x_{2}^2 - 7)^2.
    \label{eq:himmel}
\end{align}
This well-known function has four local minima and one maximum as shown in Fig.~\ref{fig:himmel}. It is frequently used to test new optimization algorithms. 
\begin{figure}
    \centering
    \includegraphics[width = 0.8\textwidth]{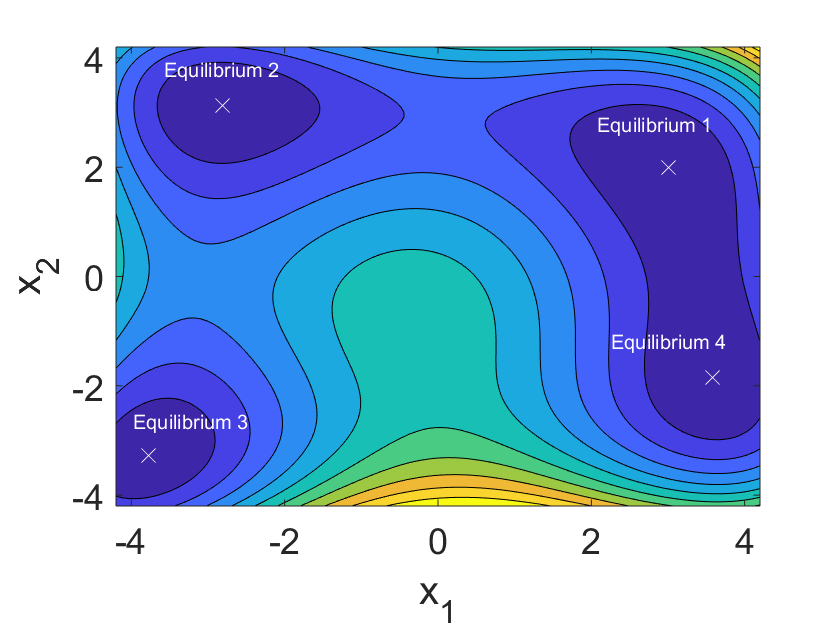}
    \caption{The cost function and equilibria in the nominal Himmelblau example.}
    \label{fig:himmel}
\end{figure}

Let us now consider a random version of the above cost function,
\begin{align}
    f(x_1,x_2) & = (x_{1}^2 + x_{2} - 11 + 2.0\lambda)^2 + (x_{1} + x_{2}^2 - 7)^2.
    \label{eq:himmel_rand}
\end{align}
where $\lambda$ is random variable drawn from the standard normal distribution. We compute the four minima for the deterministic case (Eqn.~\ref{eq:himmel}) using multiple initial conditions (one initial condition is picked in the basin of attraction for each equilibrium). The computed equilibria are depicted in Fig.~\ref{fig:himmel} and tabulated in the second column of table~\ref{tab:himmel}. 

We then introduce the uncertainty as captured in Eqn.~\ref{eq:himmel_rand}. The baseline statistics (mean and standard deviations) for the four equilibria are computed using crude Monte Carlo sampling. In particular, the statistics of every equilibrium is computed using $1000$ independent Monte Carlo samples. The computed mean and standard deviations are tabulated in table~\ref{tab:himmel}. We then use a simple first order polynomial chaos expansion for $x_{1}$ and $x_{2}$ of the form,
\begin{align}
x_{1} &\approx a_{0}\psi_{0} + a_{1}\psi_{1}\nonumber\\
x_{2} &\approx b_{0}\psi_{0} + b_{1}\psi_{1}.
\label{eq:himmel_approx}
\end{align}
\begin{table}
\centering
\begin{tabular}{|c|c|c|c|c|c| }
 \hline
\multicolumn{1}{|c|}{Equlibria} & \multicolumn{1}{|c|}{Deterministic} & \multicolumn{2}{|c|}{Random - Monte Carlo} & \multicolumn{2}{|c|}{Random - Polynomial Chaos}\\
 \hline
& & Mean & Std. Dev. & Mean & Std. Dev.\\
 \hline
Equil. $1$ & $(3.0,2.0)$ & $(2.98, 2.0)$   & $(0.36,0.09)$   & $(2.93,2.06)$&  $(0.48,0.15)$\\
 \hline
Equil. $2$& $ (-2.81,3.13) $ & $(-2.77,3.13)$ &   $(0.35,0.06)$  & $(-2.79,3.12)$   &  $(0.31,0.06)$\\
 \hline
Equil. $3$ & $(-3.78,-3.28)$ & $(-3.76,-3.28)$ &  $(0.27,0.04)$ & $(-3.77,-3.29)$ & $(0.29,0.05)$\\
 \hline
 Equil. $4$ & $(3.58,-1.85)$ & $(3.59,-1.85)$& $(0.27,0.07)$ & $(3.58,-1.83)$ & $(0.22,0.09)$\\
 \hline
\end{tabular}
\caption{Statistics of the equilibria of the Himmelblau example.}
\label{tab:himmel}
\end{table}
Using the orthonormality conditions as described in Eqn.~\ref{eq:ortho}, results in the following optimization problem,
\begin{align}
\displaystyle\min_{\left[a_{0},a_{1},b_{0},b_{1}\right]} & (a_{0}^4 + b_{0}^4) + 3(a_{1}^4 + b_{1}^4) + 6(a_{0}^2a_{1}^2 + b_{0}^2b_{1}^2) + 2(a_{0}^2b_{0} + a_{0}b_{0}^2 + a_{0}b_{1}^2) \nonumber\\
&+ 4(a_{0}a_{1}b_{1} + a_{1}b_{0}b_{1}) - 21(a_{0}^2 + a_{1}^2) - 13(b_{0}^2 + b_{1}^2) + 8a_{0}a_{1}  \nonumber\\
&+ 2a_{1}b_{0} - 14a_{0} -22b_{0} + 4b_{1} + 174.
\label{eq:himmel_pc}
\end{align}
For the polynomial chaos expansion we perform a single optimization for each equilibria by picking an initial condition in its basin of attraction. The optimal values of $\left[a_{0},a_{1},b_{0},b_{1}\right]$ are used to compute the mean and standard deviations of the equilibrium values under uncertainty. The results are shown in table~\ref{tab:himmel}. We find that in comparison to Monte Carlo, the polynomial chaos approach for stochastic optimization fares well. In particular, we find that the approach gets within $2-3$\% of the mean value with the simple linear approximation in Eqn.~\ref{eq:himmel_approx}. The standard deviation values are also very close with the exception of equilibrium $1$. We believe that the cause of this is the quadratic nature of the local gradient that will require second order expansions for $x_{1}$ and $x_{2}$.

Let us now consider optimization problems of greater complexity such as a formulation of human-machine optimization subject to constraints.
\subsection{Application to human-machine optimization}
We now consider the application of the polynomial chaos for stochastic optimization in the context of task scheduling for human-machine interaction. In particular, we consider the problem of scheduling tasks for human-UAV missions in the presence of uncertainty. Using the formulation detailed in~\cite{cit:peters}, we apply our approach to a finite horizon scheduling problem. Our goal is to optimally schedule tasks in the presence of behavioral variance that accounts for levels of arousal or stress on cognitive performance~\cite{yerkes1908relation}. For example, the Yerkes-Dodson law is frequently used to capture the notion that intermediate levels of stress or arousal give rise to the best human performance~\cite{yerkes1908relation}. This law can be imposed in the scheduling algorithm as a constraint that ensures that the \textit{task load} of an unmanned aerial vehicle (UAV) operator is kept within predetermined bounds~\cite{cit:peters}.

The overall problem can posed as a linear program by relaxing the assumption that each task must fit into a single slot as done in~\cite{cit:peters}. Let $x_{ij}$ denote the fraction of task $i$ that is performed in time (task) slot $j$ (see Fig.~\ref{Fig:task_sched}). Simplifying the formulation in~\cite{cit:peters}, each task $i$ can be specified by the following 3-tuple $\left[t_{i}, r_{i}, \delta w_{i}\right]$. Here $t_{i}$ is the amount of time it takes to complete task $i$, $r_{i}$ is the reward for completing the task, and $\delta w_{i}$ is the increase in task load of the operator due to the completion of the task.
\begin{figure}[htb!]
    \begin{center}
    \includegraphics[scale=0.45]{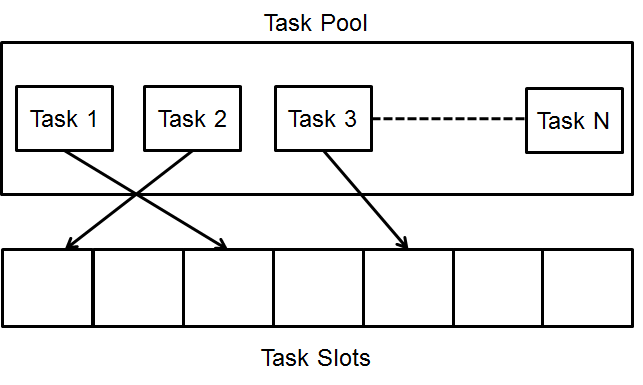}
    \end{center}
    \vspace{-0.2in}
\caption{Task scheduling depiction.}
\label{Fig:task_sched}
\end{figure}

Adapting~\cite{cit:peters}, the overall optimization problem can be posed as follows,
\begin{align}
\displaystyle &\max_{x_{ij}} \displaystyle\sum_{i}\sum_{j} r_{i}x_{ij},\nonumber \\
\text{subject to} &\displaystyle\sum_{j}x_{ij} \leq 1, \nonumber\\
                  & \displaystyle\sum_{i}\delta w_{i}x_{ij} \leq \beta,
\label{eq:human_opt}
\end{align}
where $\beta$ is the maximum task load threshold for the operator~\cite{yerkes1908relation}. We introduce uncertainty in the upper bounds of all the inequality constraints. This variability corresponds to flexibility in the requirement for completion of tasks -- in other words, we model the uncertainty on the requirement of the percentage of tasks that must be completed.

Note that in this problem formulation, one is maximizing the reward while ensuring that the task load does not exceed prescribed limits. In~\cite{cit:peters}, the authors deal with uncertainty in constraints by using a scenario based approach. This methodology involves appending the cost function with multiple samples of the constraints and adding them to the original problem formulation. This approach suffers from multiple drawbacks including the \emph{inability to deal with uncertainty in the cost function} and \emph{scalability issues due to the number of appended constraints}. In contrast, the polynomial chaos approach is able to deal with uncertainty in both cost and constraints without artificially increasing the number of constraints. However, as noted previously, the polynomial chaos approach does suffer from a separate curse of dimensionality (as shown in Eqn.~\ref{eq:curse_dim}) that constrains its application to problems with a large number of the uncertain parameters.

Let us assume that the the operator has three available tasks in this task pool with identical reward $r_{i} = 1.0$ and identical impact on operator task load $\delta w_{i} = 3.0$. We also include a ``rest'' task that reduces the overall accumulated task load by setting the associated $\delta w_{i}$ to  $-1.0$. We assume that the UAV operator has three such available ``rest'' tasks at his/her disposal. We compute the task schedule over $10$ slots. Let $\beta$ (the threshold of the task load and variability in the upper bounds of all inequality constraints) be a random variable drawn from a Gaussian distribution with mean $\mu_{0}(\beta) = 1.0$ and standard deviation $\mu_{1} = 0.2$.

We find that $1000$ Monte Carlo samples predict that the average behavior of the algorithm is to spread the six available tasks into each one of the slots, i.e. perform three tasks interspersed with three breaks. On average the linear program (using MATLAB's linprog function) is able to assign $99.48\%$ of the three tasks while using $91.47\%$ of available rest periods (spread out over $10$ slots). The variance in the completion of the tasks is $3.77\%$ with $3.05\%$ variance on the resting task. Polynomial chaos based stochastic optimization  exploits the orthogonality constraints of the polynomial expansions, as shown in Eqn.~\ref{eq:ortho}. The method predicts that on average $100\%$ of the three tasks will be completed by utilizing $92.4\%$ of the ``rest'' task. Moreover, it predicts that the variance of task completion is $4\%$ with $1\%$ variance on the resting task. As can be seen from the numerical values, polynomial chaos based method gets accuracy to the second decimal when predicting the mean and variance of the performance of the task scheduling linear program \textit{without resorting to expensive sampling based methods}.

\section{Conclusions}
Robust and stochastic optimization methods have found application in a wide variety of settings including control theory~\cite{cit:control}, system design~\cite{cit:design}, portfolio optimization~\cite{cit:portfolio}, and inventory optimization~\cite{cit:inventory} to name a few. Despite several existing algorithms, robust and stochastic optimization in non-convex settings remains an open and challenging area of critical importance.

In this work, we take early steps towards extending uncertainty quantification methods for optimization under parametric uncertainty. In particular, we use polynomial chaos based techniques for optimizing functions in the presence of uncertainty. We treat the optimization variable value as a random variable and expand it using orthogonal polynomials. Exploiting these orthonormality constraints allows one to construct a method with exponential convergence~\cite{Wiener}. Although, the approach is standard for uncertainty analysis in the presence of uncertainty, very little work has been done to exploit these methods for stochastic optimization. Our paper lays out a framework for using the polynomial chaos approach for optimizing uncertain cost functions in the presence of constraints which may also be uncertain. We include convergence proofs, derive error bounds, and study the preservation of structure (convexity and homogeneity). We then demonstrate the approach on a simple unconstrained one dimensional optimization problem, a two-dimensional non-convex problem, and a constrained optimization problem motivated from task allocation in human-machine systems. The approach is found to accurately capture the statistics (moments) of the optimizing values in an efficient manner \emph{without resorting to expensive sampling based computations}. This results in orders-of-magnitude reduction in the computational effort in finding statistics of the optimal solution in problems with low dimensional uncertainty. 

In future work, we intend to extend this approach to optimization of discontinuous functions using wavelet expansions~\cite{hybrid_uq} and construction of iterative optimization methods by extending the framework in~\cite{surana_uq}. The latter approach is expected to mitigate the curse of dimensionality associated with polynomial chaos expansions, thereby expanding its applicability.
\section{Acknowledgements}
This material is based on work supported by the US Air Force Research Lab (AFRL), Air Force Office of Scientific Research, under contract FA9550-14-C-0022 and Defense Advanced Research Projects Agency (DARPA) and Space and Naval Warfare Systems Center, Pacific (SSC Pacific) under contract number N6600118C4031.
The author thanks Dr. C. William Gear, Dr. Warren Powell, and Dr. Arvind Raghunathan for discussions and suggestions related to the work.
\bibliographystyle{unsrt}
\bibliography{Stochastic_Opt}
\end{document}